\documentclass[12pt]{article}
\usepackage{a4wide}

\usepackage[a4paper, margin=2.3cm]{geometry}
\usepackage{amsmath,amssymb,amsthm}
\usepackage{color}
\usepackage{parskip}
\usepackage{hyperref}

\newtheorem{theorem}{Theorem}[section]

\newtheorem{corollary}[theorem]{Corollary}
\newtheorem{lemma}[theorem]{Lemma}
\newtheorem{definition}[theorem]{Definition}

\newtheorem*{definition*}{Definition}

\def\F{F}

\def\RR{\mathcal{R}}
\def\S{\mathcal{S}}
\newcommand{\I}{\mathcal{I}}

\begin{document}
\title{Four-variable expanders over the prime fields}

\author{
 D. Koh\thanks{Department of Mathematics, Chungbuk National University.
    Email: {\tt koh131@chungbuk.ac.kr
}}
\and 
    H. Mojarrad\thanks{Department of Mathematics, EPFL, Lausanne. 
    Email: {\tt hossein.mojarrad@epfl.ch}}
    \and 
    T. Pham\thanks{Department of Mathematics, UCSD
    Email: {\tt v9pham@ucsd.edu}}
    \and 
    C. Valculescu\thanks{Department of Mathematics, EPFL, Lausanne. 
    Email: {\tt adrian.valculescu@epfl.ch}}
}
\date{}
\maketitle  
\begin{abstract}
Let $\mathbb{F}_p$ be a prime field of order $p>2$, and $A$ be a set in $\mathbb{F}_p$ with very small size in terms of $p$. In this note, we show that the number of distinct cubic distances determined by points in $A\times A$ satisfies
\[|(A-A)^3+(A-A)^3|\gg |A|^{8/7},\]
which improves a result  due to Yazici, Murphy, Rudnev, and Shkredov. In addition, we investigate some new families of expanders in four and five variables.  We also give an explicit exponent of a problem of Bukh and Tsimerman, namely, we prove that
 \[\max \left\lbrace |A+A|, |f(A, A)|\right\rbrace\gg |A|^{6/5},\]
 where $f(x, y)$ is a quadratic polynomial in $\mathbb{F}_p[x, y]$ that is not of the form $g(\alpha x+\beta y)$ for some univariate polynomial $g$. 
\end{abstract}
\section{Introduction}
Let $p>2$ be a prime, and $\mathbb{F}_p$ be the finite field of order $p$. We denote the set of non-zero elements in $\mathbb{F}_p$ by $\mathbb{F}_p^*$. We say that a $k$-variable function $f(x_1,\ldots, x_k)$ is an \emph{expander} if there are $\alpha>1,~\beta>0$ such that for any sets $A_1,\ldots,A_k\subset\mathbb{F}_p$ of size $N\ll p^\beta,$ 
\[|f(A_1\times \cdots \times A_k)|\gg N^\alpha.\]
We write $X\gg Y$ if $X\ge CY$ for some positive constant $C$.

As far as we know, there are few known results on two-variable expanders. For example, it has been shown by Yazici, Murphy, Rudnev, and Shkredov \cite{AMRS} that the polynomial $f(x, y)=x+y^2$ is an expander. More precisely, they proved that if $A\subset \mathbb{F}_p$ with $|A|\le p^{5/8}$, then 
\[|A+A^2|\gg |A|^{11/10}.\]
The authors in \cite{AMRS} also indicated that the polynomial $f(x, y)=x(y+1)$ is an expander. In particular, they established that $|A\cdot (A+1)|\gg |A|^{9/8}$.

These exponents have been improved in recent works. For instance, Stevens and de Zeeuw \cite{SZ} showed that $|A\cdot (A+1)|\gg |A|^{6/5},$ and Pham, Vinh, and de Zeeuw \cite{P} proved that $|A+A^2|\gg |A|^{6/5}.$

Another expander in two variables has been investigated by Bourgain \cite{bo}. He proved that if $A, B\subset \mathbb{F}_p$ with $|A|=|B|=N=p^{\epsilon}$, $\epsilon>0$, and $f(x, y)=x^2+xy$, then $|f(A, B)|\gg N^{1+\delta},$ for some $\delta>0$. An explicit exponent was given by Stevens and de Zeeuw \cite{SZ}, namely, they proved that $|f(A, B)|\gg N^{5/4}$ for $N\le p^{2/3}$. We refer the reader to \cite{bukk, hlc, tao} and references therein for two-variable expanders in large sets over arbitrary finite fields. 

For three-variable expanders, there are several results which have been proved in recent years. Roche-Newton, Rudnev, and Shkredov \cite{RRS} proved that 
\begin{equation}\label{eqqqq:1}|A\cdot (A+A)|\gg |A|^{3/2}, ~~ |A+A\cdot A|\gg |A|^{3/2},\end{equation}
when $|A|\le p^{2/3}$.

In \cite{P}, Pham, Vinh and de Zeeuw  obtained a more general result. More precisely, they showed that for $A, B, C\subset \mathbb{F}_p$ with $|A|=|B|=|C|=N\le p^{2/3}$, and  for any quadratic polynomial in three variables $f(x, y, z)\in \mathbb{F}_p[x, y, z]$ which is not of the form $g(h(x)+k(y)+l(z))$, we have 
\begin{equation}\label{eqqqq:2}|f(A, B, C)|\gg N^{3/2}.\end{equation}
We notice that one can use the inequalities (\ref{eqqqq:1}) and (\ref{eqqqq:2}) to obtain some results on expanders in four variables. To see this, observe that  the following estimates follow directly from (\ref{eqqqq:1}) and (\ref{eqqqq:2}):
\[|(A-A)\cdot (A-A)|\gg |A|^{3/2}, ~ |A\cdot A+A\cdot A|\gg |A|^{3/2}, ~|(A-A)^2+(A-A)^2|\gg |A|^{3/2}.\]
A stronger version of the last inequality can be found in \cite{khoangcach}.  We refer the reader to \cite{mm} for a recent improvement on the size of $(A-A)\cdot (A-A)$.

In this note, we extend the methods from \cite{AMRS, RRS, P} to study different expanders in four variables over $\mathbb{F}_p$. 

Yazici et al. \cite{AMRS} proved that if $A\subset \mathbb{F}_p$ with $|A|\le p^{7/12}$, then the number of distinct cubic distances is at least $|A|^{36/35}$. Our first theorem is an improvement of this result.
\begin{theorem}\label{co0}
Let $A\subset \mathbb{F}_p$ with $|A|\le p^{7/12}$. Then we have
\[|(A-A)^3+(A-A)^3|\gg  |A|^{8/7}.\]
\end{theorem}
In our next two theorems, we provide two more expanders in four variables.
\begin{theorem}\label{mu4}
Let $A$ be a set in $\mathbb{F}_p$ with $|A|\le p^{5/8}$, $f(x)\in \mathbb{F}_p[x]$ be a quadratic polynomial, and $g(x, y)\in \mathbb{F}_p[x, y]$ be  a quadratic polynomial with a non-zero $xy$-term. Then we have
\[ |f(A)+A+g(A, A)|\gg |A|^{8/5}. \]
\end{theorem}
\begin{theorem}\label{mu5}
Let $1\le N\le p^{4/7}$ be an integer,  $h$ be a generator of $\mathbb{F}_p^*$, and $g(x, y)\in \mathbb{F}_p[x, y]$ be  a quadratic polynomial with a non-zero $xy$-term.  Then we have
\[\left\vert \{h^x+h^y+g(z, t)\colon 1\le x, y, z,t \le N\}\right\vert \gg N^{7/4}.\]
\end{theorem}
\bigskip
Different families of expanders with superquadratic growth have been studied in recent literature. For instance, Balog, Roche-Newton and Zhelezov \cite{balog} showed that for any $A\subset \mathbb{R}$ we have $|(A-A)\cdot(A-A)\cdot(A-A)|\gg |A|^{2+\frac{1}{8}}/\log^{\frac{17}{16}}|A|$. Murphy, Roche-Newton and Shkredov \cite{mur} proved that for $A\subset \mathbb{R}$ we have $|(A+A+A+A)^2+\log A|\gg |A|^2/\log |A|$. In the following theorem, we obtain two more expanders in five variables with quadratic growth.
\begin{theorem}\label{mu6}
Let $\mathbb{F}_p$ be a prime field of order $p$. Suppose that $h$ is a generator of $\mathbb{F}_p^*$, and $1\le N\le p^{1/2}$ is an integer. Then the following two statements hold:
\begin{enumerate}
\item $\left\vert \left\lbrace h^x(h^y+h^z+h^t+h^v)\colon 1\le x, y, z, t, v\le N \right\rbrace\right\vert\gg N^{2},$
\item $\left\vert \left\lbrace x(h^y+h^z+h^t+h^v)\colon 1\le x, y, z, t, v\le N \right\rbrace\right\vert\gg N^{2}.$
\end{enumerate}
\end{theorem}
\bigskip
For $A\subset \mathbb{F}_p$, the \emph{sumset} of $A$ is the set $A + A = \{ a + b : a ,b \in A \}$,
and the \emph{product set} of $A$ is the set $A \cdot A = \{ a \cdot b : a , b \in A \}$. In $2004$, Bourgain, Katz, and Tao \cite{bkt} proved that if $p^{ \delta } < |A| < p^{1- \delta}$ 
where $0 < \delta < 1/2$, then 
\begin{equation}\label{eq:bkt}
\max \{ |A + A| , |A \cdot A| \} \geq c  |A|^{1 + \epsilon }
\end{equation}
for some positive constants $c$ and $\epsilon$ depending only on $\delta$.  Hart, Iosevich, and Solymosi \cite{his} obtained bounds that give an explicit dependence of $\epsilon$ on $\delta$.  
In \cite{his}, it is shown that if 
$|A + A| = m$ and $|A \cdot A | = n$, then 
\begin{equation}\label{eq:his}
|A|^3 \leq \frac{ c m^2 n |A| }{ p} + c p^{1/2} mn
\end{equation}
where $c$ is some positive constant.  Inequality (\ref{eq:his}) implies a non-trival sum-product estimate when 
$p^{1/2} \ll |A| \ll p $. Vinh \cite{vinh} and Garaev \cite{garaev} improved the inequality (\ref{eq:his}) and as a result, obtained a better sum-product estimate.     

\begin{theorem}[\cite{vinh}]\label{vinh th}
For  $A \subset \mathbb{F}_p$, suppose that $|A+A| = m$, and $|A \cdot A| = n$, then 
\[
|A|^2 \leq \frac{mn|A|}{p} + p^{1/2} \sqrt{ m n } .
\]
\end{theorem}

\begin{corollary}[\cite{vinh}]\label{vinh cor}
For $A \subset \mathbb{F}_p$, then there is a positive constant $c$ such that the following hold.  
\begin{enumerate}
\item If $p^{1/2} \ll |A| < p^{2/3}$, then 
\[
\max \{ |A + A | , |A \cdot A | \} \geq \frac{ c |A|^2 }{p^{1/2}} .
\]
\item If $p^{2/3} \leq |A| \ll p$, then 
\[
\max \{ |A + A | , |A \cdot A | \} \geq  c ( p |A| )^{1/2} .
\]
\end{enumerate}
\end{corollary}
A more general statement  of Corollary \ref{vinh cor} has been established by Vu \cite{vu}. Before presenting his result, we need the following definition. 
 \begin{definition} \label{def:deg} A polynomial $f(x, y) \in \mathbb{F}_p[x, y]$ is {\it degenerate}
 if it is  of the form
$Q(L(x, y))$ where $Q$ is an one-variable polynomial and $L$ is
a linear form in $x$ and  $y$. \end{definition}
Vu \cite{vu} proved the following theorem.
\begin{theorem}[\cite{vu}] \label{theorem:1} Let
$f(x, y)$  be a non-degenerate polynomial of degree $d$ in
$\mathbb{F}_p[x, y]$. Then for any $A \subset \mathbb{F}_p$, we have

$$\max \left\lbrace |A+A|, |f(A, A)| \right\rbrace \gg \min \left\lbrace \frac{|A|^{3/2}}{dp^{1/4}},  \frac{p^{1/3}|A|^{2/3}}{d^{1/3}} \right\rbrace.
$$
\end{theorem}
We note that, in the case $f(x, y)=xy$, the lower bounds of Theorem \ref{theorem:1} are weaker than those of Corollary \ref{vinh cor}. Theorem \ref{theorem:1} is only non-trivial when $|A|\gg p^{1/2}$, and Theorem \ref{theorem:1} also holds over arbitrary finite fields $\mathbb{F}_q$ with $q$ is a prime power. The reader can find a version of Theorem \ref{theorem:1} over the real numbers in \cite{shen}. When $|A|\le \sqrt{p}$ and $f(x, y)$ is a non-degenerate quadratic polynomial, Bukh and Tsimerman \cite{bukk} obtained the following improvement.
\begin{theorem}[\cite{bukk}]
Let $f(x, y)\in \mathbb{F}_p[x, y]$ be a  non-degenerate quadratic polynomial.  For any $A\subset \mathbb{F}_p$ with $|A|\le \sqrt{p}$, we have
\begin{equation}\label{eqqq!}\max\{|A+A|, |f(A, A)|\}\gg |A|^{1+\epsilon},\end{equation}
for some $\epsilon>0$.
\end{theorem}
There are several progresses on finding explicit exponents of the inequality (\ref{eq:bkt}) for small sets over recent years, and the best lower bound was given by Roche-Newton, Rudnev, and Shkredov \cite{RRS}. More precisely, they showed that for $A\subset \mathbb{F}_p$ with $|A|\le p^{5/8}$,  the sum set and the product set satisfy
\[\max\left\lbrace |A+ A|, |A\cdot A|\right\rbrace \gg |A|^{6/5}.\]
In this paper, we give an explicit exponent of the inequality (\ref{eqqq!}) as follows.
\begin{theorem}\label{maymay}
Let $f(x, y)\in \mathbb{F}_p[x,y]$ be a non-degenerate quadratic polynomial. Let $A$ be a set in $\mathbb{F}_p$ with $|A|\le p^{5/8}$, then we have 
\[\max\{|A+A|, |f(A, A)|\}\gg |A|^{6/5}.\]
\end{theorem}
\bigskip

The rest of this paper is organized as follows. In Section $2$, we mention main tools in our proofs. We give a proof of Theorem \ref{co0} in Section $3$. Proofs of Theorems \ref{mu4}, \ref{mu5}, and \ref{mu6} are given in Section $4$. In Section $5$ we will give a proof of Theorem \ref{maymay}, and a discussion on an improvement of Theorem \ref{theorem:1} for large sets.
\section{Tools}
The main tool in our proofs is a point-plane incidence bound due to Rudnev \cite{R}, but
we use a strengthened version of this theorem, proved by de Zeeuw in \cite{Z}. Let us first recall that if $\RR$ is a set of points in $\mathbb{F}_p^3$ and $\S$ is a set of planes in $\mathbb{F}_p^3$, then the number of incidences between $\RR$ and $\S$, denoted by $I(\RR, \S)$, is the cardinality of the set $\{(r,s)\in \RR\times \S : r\in s\}$.
\begin{theorem}[{\bf Rudnev}, \cite{R}]\label{thm:rudnev}
Let $\RR$ be a set of points in $\mathbb{F}_p^3$ and $\S$ be a set of planes in $\mathbb{F}_p^3$, with $|\RR|\leq |\S|$ and $|\RR|\ll p^2$.
Suppose that there is no line that contains $k$ points of $\RR$ and is contained in $k$ planes of $\S$.
Then
\[ \I(\RR,\S)\ll |\RR|^{1/2}|\S| +k|\S|.\]
\end{theorem}
The following lemma is known as the Pl\"unnecke-Ruzsa inequality. A simple and elegant proof can be found in \cite{pepe}.
\begin{lemma}[{\bf Pl\"unnecke-Ruzsa}]\label{l:plunnecke-ruzsa}
Let  $A, B$ be finite subsets of an abelian group such that $|A+B|\le K|A|.$
Then, for an arbitrary $0<\delta<1$, there is a nonempty set $X\subset A$ such that
$|X| \ge (1-\delta) |A|$ and for any integer $k$ one has
\begin{equation}\label{f:plunnecke-ruzsa1'}
    |X+kB|\le \left(\frac{K}{\delta} \right)^{k}  |X| \,.
\end{equation}
\end{lemma}

To prove Theorems \ref{mu4}--\ref{mu6}, we need the following two lemmas. The first one follows from a result of Pham, Vinh and de Zeeuw \cite{P}.

\begin{lemma}\label{mot}
Let $g(x, y)\in \mathbb{F}_p[x,y]$ be a quadratic polynomial with a non-zero $xy$-term.
Let $A,X\subset \F$ with $|A|\leq |X|$.
Then we have
\[|g(A, A)+X|\gg \min\left\{|A||X|^{1/2},p\right\}. \]
\end{lemma}
The second lemma we use is due to Yazici et al. and was proved in \cite{AMRS}.
\begin{lemma}\label{ba}If $A, X\subset \mathbb{F}_p$ with $|X|\le |A|$, then
\[|X\cdot (A-A)|\gg \min\left\lbrace |A||X|^{1/2}, p\right\rbrace.\]
\end{lemma}

\section{Proof of Theorem \ref{co0}}
We need the following result in order to prove Theorem \ref{co0}. 
\begin{lemma}\label{thm2*}
Let $A, X\subset \mathbb{F}_p$ with $|A-A|^2|X|\le p^2$. Then
\[|\left\lbrace (b-a)^3+a^3+x\colon a, b\in A, x\in X\right\rbrace|\gg  \min \left\lbrace  \frac{|X|^{1/2}|A|^4}{|A-A|^3},  \frac{|X||A|^5}{|A-A|^4}\right\rbrace.\]
\end{lemma}
\begin{proof}
First note that
\begin{align}\label{eq1}
(b-a)^3+a^3&=3b\left((a-b/2)^2+b^2/12\right)=3b(t^2+b^2/12)
\end{align}
where $t=a-b/2$.
Define $T=\{a-b/2\colon a, b\in A\}$, and let $E$ be the number of solutions of the following equation
\[(b-a)^3+a^3+x=(b'-a')^3+a'^3+x', ~~~a, a', b, b'\in A, x, x'\in X.\]
To bound $E$, we first define a set of points $\RR$ and a set of planes $\S$ as follows:
\[\mathcal{R}=\{(t^2, b', -b'^3/4+x)\colon t\in T, b'\in A, x\in X\}, \]
\[\mathcal{S}=\{3bX-3t'^2Y+Z=-b^3/4+x'\colon t'\in T, x'\in X, b\in A\}.\]
It is clear that $|\mathcal{R}|=|\mathcal{S}|\ll |T||A||X|$, and $|T|\le |A+A-A|$.

Lemma \ref{l:plunnecke-ruzsa} implies that for any $0<\delta<1$, there exists a nonempty set $A'\subset A$ with $|A'|\ge (1-\delta)|A|$ satisfying
\[|A+A-A'|\ll \frac{|A-A|^2}{|A|}.\]
Since we can choose $\delta$ such that $|A'|=\Theta(|A|)$ \footnote{$X=\Theta(Y)$ means that there exist positive constants $C_1$ and $C_2$ such that $C_1Y\le X\le C_2Y$}, we can assume that $|T|\ll \frac{|A-A|^2}{|A|}.$ This implies that
\[|\mathcal{R}|, |\mathcal{S}|\ll |A-A|^2|X|.\]
By the assumption, we have $|\RR| \ll p^2$.
This allows us to apply Theorem \ref{thm:rudnev},
 assuming we can prove an upper bound on the maximum number $k$ for which there is a line that contains $k$ points of $\RR$ and is contained in $k$ planes of $\S$. The projection of $\RR$ onto the first two coordinates is $\{t^2: t \in T\}\times A$,
so each line contains at most $\max\{|A|,|T|\}$ points of $\RR$, unless it is vertical, in which case it could contain $|X|$ points of $\RR$.
However, the planes in $\mathcal{S}$ contain no vertical lines, so in this case the hypothesis of Theorem \ref{thm:rudnev} is satisfied with $k = \max\{|A|,|T|\}\ll |A-A|^2/|A|$.

Therefore, Theorem \ref{thm:rudnev} implies that
\[E\ll |X|^{3/2}|A-A|^3+|X||A-A|^4/|A|.\]
By the Cauchy-Schwarz inequality, we have 
\[|\{(b-a)^3+a^3+x\colon a, b \in A, ~x\in X\}|\gg \frac{|A|^{4}|X|^2}{E}\gg \min \left\lbrace  \frac{|X|^{1/2}|A|^4}{|A-A|^3},  \frac{|X||A|^5}{|A-A|^4}\right\rbrace.\]
This completes the proof of the lemma.
\end{proof}
\begin{proof}[Proof of Theorem \ref{co0}]
Since the cubic distance function is invariant under translations, we assume that $0\in A$. It follows from the Pl\"unnecke-Ruzsa inequality that there exists a set $X\subset (A-A)^3$ with $|X|=\Theta(|(A-A)|)$ such that
\[|X+(A-A)^3+(A-A)^3|=|X+2(A-A)^3|\ll \frac{|(A-A)^3+(A-A)^3|^2}{|(A-A)^3|^2}|(A-A)^3|.\]
This implies that 
\[|(A-A)^3+(A-A)^3|^2\gg |A-A|\cdot |X+(A-A)^3+(A-A)^3|.\]
On the other hand, if $|A-A|^2|X|> p^2$, then we have $|A-A|\gg p^{2/3}$. This implies that $|A-A|\gg |A|^{8/7}$ since $|A|\le p^{7/12}$, and we are done. Thus, we may assume $|A-A|^2|X|\le p^2$, and it follows from Lemma \ref{thm2*} that 
\[|X+(A-A)^3+(A-A)^3|\gg \frac{|A|^4}{|A-A|^{5/2}}.\]
Therefore, we obtain
\[|(A-A)^3+(A-A)^3|^2\gg \frac{|A|^4}{|A-A|^{3/2}},\]
which leads to  
\[\max\left\lbrace |(A-A)^3+(A-A)^3|, |A-A|\right\rbrace\gg |A|^{8/7}.\]
This concludes the proof of the theorem.
\end{proof}
\section{Proofs of Theorems \ref{mu4}, \ref{mu5}, and \ref{mu6}}
We use of the following lemmas in the proofs of Theorems \ref{mu4}-\ref{mu6}.
\begin{lemma}\label{buc1}
Let $f(x)\in \mathbb{F}_p[x]$ be a quadratic polynomial. For $A\subset \mathbb{F}_p$ with $|A|\le p^{5/8}$, we have
\[|f(A)+A|\gg |A|^{6/5}.\]
\end{lemma}
\begin{proof}
Without loss of generality, we can assume that $f(x)=ax^2+bx$ with $a\ne 0$. Consider the following equation
\begin{equation}\label{eq:eq1}a(x-y)^2+b(x-y)+z=t,
\end{equation}
with $x\in A+f(A)$, $y\in f(A)$, $z\in A$, and $t\in A+f(A)$. Since $f$ is a quadratic polynomial, we have $|f(A)|=\Theta(|A|)$.

Note that for any $u,v,w\in A$, a solution of \eqref{eq:eq1} is given by $x =u+f(v)\in A+f(A)$, $y = f(v)\in f(B)$, $z = w\in A$, and $t = w+f(u)\in A+f(A)$.
Therefore, we have
\begin{equation}\label{eq:eq2}
|A|^3\le\left|\left\{(x,y,z,t)\in (A+f(A))\times f(A)\times A\times (A+f(A)): a(x-y)^2+b(x-y)+z=t\right\}\right|.
\end{equation}
If we define $E$ to be the cardinality of the following set
\[\left\lbrace (x,y,z,x',y',z')\in ((A+f(A))\times f(A)\times A)^2\colon f(x-y)+z=f(x'-y')+z'\right\rbrace, \]
then \eqref{eq:eq2} together with the Cauchy-Schwarz inequality give
\begin{equation}\label{eq:eq3}
\frac{|A|^6}{|A+f(A)|}\ll E.
\end{equation}

To bound $E$, we use Theorem \ref{thm:rudnev} for the following point set
\[ \mathcal{R} = \{(ax,y', bx+ax^2+z-a(y')^2+by'): (x,y',z)\in (A+f(A))\times f(A)\times A\}\]
and the following set of planes
\[ \mathcal{S} = \{-2yX + 2ax'Y + Z = a(x')^2+bx'+z' - ay^2+by: (x',y,z')\in (A+f(A))\times f(A)\times A\}.\]
Note that if $|A+f(A)|\gg |A|^{6/5}$, then we are already done. Therefore, we can assume that $|A+f(A)|\ll |A|^{6/5}$, from which we obtain $|\mathcal{R}|=|A+f(A)||f(A)||A|\ll |A|^{16/5}\ll p^2$, since $|A|\ll p^{5/8}$. The projection of $\RR$ onto the first two coordinates is $(A+f(A))\times f(A)$,
so each line contains at most $\max\{|A+f(A)|, |f(A)|\} = |A+f(A)|$ points of $\RR$, unless it is vertical, in which case it may contain $|A|$ points of $\RR$.
However, the planes in $\mathcal{S}$ contain no vertical lines, so in this case the hypothesis of Theorem \ref{thm:rudnev} is satisfied with $k = |A+f(A)|$. Thus, Theorem \ref{thm:rudnev} implies that 
\begin{equation}\label{eq:eq4}
E\ll I(\mathcal{R}, \mathcal{S})\ll  |A+f(A)|^{3/2}|A|^3 + |A+f(A)|^2|A|^2.
\end{equation}
If $|A+f(A)|^2|A|^2$ is asymptotically larger than $|A+f(A)|^{3/2}|A|^3$, then $|A+f(A)| \gg |A|^2$, so we are done.
Otherwise, we can assume that $|A+f(A)|^{3/2}|A|^3$ is bigger than $|A+f(A)|^2|A|^2$,
so combining \eqref{eq:eq3} and \eqref{eq:eq4} gives
\[\frac{|A|^6}{|A+f(A)|} \ll |A+f(A)|^{3/2}|A|^3,\]
which leads to
\[|f(A)+A|\gg |A|^{6/5}.\]
This completes the proof of the lemma.
\end{proof}
\begin{lemma}\label{mu1}
Let $\mathbb{F}_p$ be a prime field of order $p$, and suppose that $h$ is a generator of $\mathbb{F}_p^*$, and $1\le N\le p^{2/3}$ is an integer. Then
\[\left\vert \left\lbrace h^x+h^y\colon 1\le x, y\le N \right\rbrace\right\vert\gg N^{3/2}.\]
\end{lemma}
\begin{proof}
Define $A:=\{h^x\colon 1\le x\le N/2\}$, and $X:=\{h^x\colon 1\le x\le N\}$. Then one can check that
\[|\left\lbrace h^x+h^y\colon 1\le x, y\le N\right\rbrace|\gg |A\cdot A+X|.\]
Thus the lemma follows directly from Lemma \ref{mot}.
\end{proof}
\begin{proof}[Proofs of Theorems \ref{mu4}, \ref{mu5}, and \ref{mu6}]
Theorem \ref{mu4} follows from Lemmas \ref{mot} and \ref{buc1}. Theorem \ref{mu5} follows directly from Lemmas \ref{mot} and \ref{mu1}.
Theorem \ref{mu6} follows from Lemmas \ref{ba} and \ref{mu1}.
\end{proof}
\section{Proof of Theorem \ref{maymay}}
To prove Theorem \ref{maymay}, we use the following lemma, which follows directly from Lemmas $2.2$ and $2.3$ in \cite{P}. We refer the reader to \cite{P} for a detailed proof.
\begin{lemma}\label{tongtong}
Let $f(x, y, z)\in \mathbb{F}_p[x,y,z]$ be a quadratic polynomial that depends on each variable and is not of the form $g(h(x)+k(y)+l(z))$. Let $A, B,C\subset \mathbb{F}_p$ with $|A|=|B|\le |C|$ and $|A||B||C|\ll p^2$. Then we have 
\[\left\vert \left\lbrace(x, y, z, x', y', z')\in (A\times B\times C)^2\colon f(x, y, z)=f(x', y', z')\right\rbrace\right\vert\le (|A||B||C|)^{3/2}+|A||B||C|^2.\]
\end{lemma}
We are now ready to give a proof of Theorem \ref{maymay}.
\begin{proof}[Proof of Theorem \ref{maymay}]
Without loss of generality, we assume that  $f(x, y)=ax^2+by^2+cxy+dx+ey$ with $a\ne 0$. Let $f'(x, y, z):=f(z-x, y)$. Consider the following equation
\begin{equation}\label{eqmay}f'(x, y, z)=t, \end{equation}
with $x\in A, y\in A, z\in A+A, t\in f(A, A)$. 

Note that for any $u,v,w\in A$, a solution of \eqref{eqmay} is given by $x =u\in A$, $y = v\in A$, $z = u+w\in A+A$, and $t = f(w, v)\in f(A, A)$.
Thus, we have
\begin{equation}\label{eq:eq2may}
|A|^3\ll\left|\left\{(x,y,z,t)\in A\times A\times (A+A)\times f(A, A): f'(x, y, z)=t\right\}\right|.
\end{equation}
Let $E$ be the cardinality of the following set
\[\left\lbrace (x,y,z,x',y',z')\in (A\times A\times (A+A))^2\colon f'(x, y, z)=f'(x', y', z')\right\rbrace. \]
Then \eqref{eq:eq2may} and the Cauchy-Schwarz inequality give
\begin{equation}\label{eq:eq3may}
\frac{|A|^6}{|f(A, A)|}\ll E.
\end{equation}
Before applying Lemma \ref{tongtong}, we need to show that $f'(x, y, z)$ is not of the form $g'(h'(x)+k'(y)+l'(z))$. By the contradiction, suppose $f'(x, y, z)=g'(h'(x)+k'(y)+l'(z))$. Then $g'$ is a polynomial of degree $2$ since $a\ne 0$. Thus, $h'$, $k'$, and $l'$ are linear polynomials. So we can write $f'(x, y, z)$ as
\[f'(x, y, z)=g'(\lambda_1x+\lambda_2y+\lambda_3z+\lambda_4),\]
for some $\lambda_1, \lambda_2, \lambda_3, \lambda_4 \in \mathbb{F}_p$. Since $g'$ is a polynomial of degree $2$, without loss of generality, we assume that $g'(x)=x^2+\lambda_5x+\lambda_6$ for some $\lambda_5, \lambda_6\in \mathbb{F}_p$. It follows from the definition of $f'$ that
\[\lambda_1^2=\lambda_3^2=a, ~2\lambda_1\cdot\lambda_3=-2a.\] 
This implies that $\lambda_1=-\lambda_3$. Hence, $f'$ can be presented as 
\[f'(x, y, z)=g'(\lambda_3(z-x)+\lambda_2 y+\lambda_4).\]
From here, we can rearrange the coefficients of $g'$ such that $g'(\lambda_3(z-x)+\lambda_2 y+\lambda_4)=g''(\lambda_3(z-x)+\lambda_2 y)$ for some $g''\in \mathbb{F}_p[x]$. This leads to 
\[f(z-x, y)=g''(\lambda_3(z-x)+\lambda_2 y),\]
which contradicts the assumption of the theorem. 

In other words, we have that $f'(x, y, z)$ is not of the form $g'(h'(x)+k'(y)+l'(z))$. 

If $|A|^2|A+A|\gg p^2$, then we have $|A+A|\gg |A|^{6/5}$ since $|A|\le p^{5/8}$, and we are done. Thus we can assume that $|A|^2|A+A|\ll p^2$. Lemma \ref{tongtong} with $B=A$ and $C=A+A$ implies that
\[E\le |A|^3|A+A|^{3/2}+|A|^2|A+A|^2.\]
Therefore, the theorem follows from the inequality (\ref{eq:eq3may}).
\end{proof}
We note that if we use the point-plane incidence bound due to Vinh \cite{vinh} for large sets in the proofs of Lemmas $2.2$ and $2.3$ in \cite{P}, then we are able to obtain the following version of Lemma \ref{tongtong} for large sets.
\begin{lemma}\label{tongtong1}
Let $\mathbb{F}_q$ be an arbitrary finite field.  Let $f(x, y, z)\in \mathbb{F}_q[x,y,z]$ be a quadratic polynomial that depends on each variable and is not of the form $g(h(x)+k(y)+l(z))$. Let $A, B,C\subset \mathbb{F}_q$, then we have 
\[\left\vert \left\lbrace(x, y, z, x', y', z')\in (A\times B\times C)^2\colon f(x, y, z)=f(x', y', z')\right\rbrace\right\vert\le \frac{(|A||B||C|)^2}{q}+q|A||B||C|.\]
\end{lemma}
One can follow identically the proof of Theorem \ref{maymay} with Lemma \ref{tongtong1} to obtain the following improvement of Vu's result for quadratic polynomials. We leave the detailed proof to the reader.
\begin{theorem}
Let $\mathbb{F}_q$ be an arbitrary finite field.  Let $f(x, y)\in \mathbb{F}_q[x,y]$ be a non-degenerate quadratic polynomial. Let $A$ be a set in $\mathbb{F}_q$, then we have 
\[\max\{|A+A|, |f(A, A)|\}\gg \min\left\lbrace \frac{|A|^2}{q^{1/2}}, q^{1/3}|A|^{2/3}\right\rbrace.\]
\end{theorem}
\section*{Acknowledgement}
The authors would like to thank Frank de Zeeuw for useful discussions and comments.

The first listed author was supported by Basic Science Research Program through the National Research Foundation of Korea(NRF) funded by the Ministry of Education, Science and Technology(NRF-2015R1A1A1A05001374). The second listed author was supported by Swiss National Science Foundation grant P2ELP2175050. The third, and fourth listed authors were partially supported by Swiss National Science Foundation grants 200020-162884 and 200021-175977.


\begin{thebibliography}{1}
\providecommand{\url}[1]{\texttt{#1}}
\expandafter\ifx\csname urlstyle\endcsname\relax
  \providecommand{\doi}[1]{doi: #1}\else
  \providecommand{\doi}{doi: \begingroup \urlstyle{rm}\Url}\fi

\bibitem{AMRS}
E. Aksoy Yazici, B. Murphy, M. Rudnev, and I. Shkredov,
{\em Growth estimates in positive characteristic via collisions},
to appear in International Mathematics Research Notices.
Also in {\tt arXiv:1512.06613}, 2015.
\bibitem{Ya}
E. Aksoy Yazici, \textit{Sum-Product Type Estimates for Subsets of Finite Valuation Rings}, arXiv:1701.08101, 2016.
\bibitem{bkt}
J.\ Bourgain, N.\ Katz, T.\ Tao,
{\em A sum-product estimate in finite fields, and applications},
GAFA 14 (2004) 27-57.
\bibitem{balog}
A. Balog, O. Roche-Newton, D. Zhelezov, \textit{Expanders with superquadratic growth}, arXiv:1611.05251v1, 2016.
\bibitem{bukk}
B. Bukh, J. Tsimerman, \emph{Sum–product estimates for rational functions}, Proceedings of the London Mathematical Society, 104(1) (2012), 1-26.
\bibitem{bo}
J. Bourgain, \textit{More on the sum-product phenomenon in prime fields and its applications},
International Journal of Number Theory \textbf{1} (2005), 1--32.
\bibitem{garaev}
M.\ Z.\ Garaev, 
{\em The sum-product estimate for large subsets of prime fields},
Proc.\ Amer.\ Math.\ Soc.\ 136 (2008) 2735-2739.
\bibitem{his}
D.\ Hart, A.\ Iosevich, J.\ Solymosi,
{\em Sum-product estimates in finite fields via Kloosterman sums},
Int.\ Math.\ Res.\ Not.\  no.\ 5, (2007) Art.\ ID rnm007.   
\bibitem{hlc}
D. Hart, L. Li, and C.-Y. Shen, \textit{Fourier analysis and expanding phenomena in finite fields}, Proceedings of the American Mathematical Society, \textbf{141}(2) (2013), 461--473.
\bibitem{mm}
B. Murphy, G. Petridis, O. Roche-Newton, M. Rudnev, I. D. Shkredov, \textit{New results on sum-product type growth over fields}, arXiv:1702.01003, 2017.
\bibitem{mur}
B. Murphy, O. Roche-Newton, I. Shkredov, \textit{Variations on the sum-product problem II}, arXiv:1703.09549, 2017.
 \bibitem{P}
 T. Pham, L. A. Vinh,  F. de Zeeuw, \textit{Three-variable expanding polynomials
and higher-dimensional distinct distances}, arXiv:1612.09032, 2016.
\bibitem{pepe}
G. Petridis, \textit{New proofs of Plünnecke-type estimates for product sets in groups}, Combinatorica,\textbf{ 32}(6) (2012), 721--733.
\bibitem{khoangcach}
G. Petridis,  \textit{Pinned algebraic distances determined by Cartesian products in $\mathbb{F}_p^2$}, to appear in Proceedings of the American Mathematical Society, 2017.
\bibitem{RRS}
O. Roche-Newton, M. Rudnev, and I.D. Shkredov,
{\em New sum-product type estimates over finite fields}, Advances in Mathematics \textbf{293} (2016), 589--605.
\bibitem{R}
M. Rudnev,
{\em On the number of incidences between points and planes in three dimensions},
to appear in Combinatorica.
Also in
{\tt arXiv:1407.0426}, 2014.
\bibitem{shen}
C-Y, Shen, \textit{Algebraic methods in sum-product phenomena}, Israel Journal of Mathematics, \textbf{188}(1) (2012), 123--130.
\bibitem{SZ}
S. Stevens and F. de Zeeuw,
\emph{An improved point-line incidence bound over arbitrary fields},
{\tt arXiv:1609.06284}, 2016.
\bibitem{tao}
 T. Tao, \textit{The sum-product phenomenon in arbitrary rings}, Contributions to Discrete
Mathematics, \textbf{4} (2), 2009.
\bibitem{v}
 L. A. Vinh,  \textit{On four-variable expanders in finite fields}, SIAM J. Discrete Math.,
\textbf{27}(4):2038--2048, 2013.

\bibitem{vinh}
L.\ A.\ Vinh, 
{\em The Szemer\'{e}di-Trotter type theorem and the sum-product estimate in finite fields},
Euro.\ J.\ Combin.\ 32 (2011), no.\ 8, 1177-1181.
\bibitem{vu}
V.\ Vu,
{\em Sum-product estimates via directed expanders},
Math.\ Res.\ Lett.\ {\bf 15} (2008), no.\ 2, 375-388.

\bibitem{Z}
F. de Zeeuw, {\em A short proof of Rudnev's point-plane incidence bound},
{\tt arXiv:1612.02719}, 2016.
\bibitem{zhe}
D. Zhelezov,  \textit{On additive shifts of multiplicative almost-subgroups in finite fields}, to appear in Proceedings of the American Mathematical Society, 2017.

\end{thebibliography}
\end{document}